\documentclass[11pt, leqno]{amsart}
\usepackage[dvipsnames]{xcolor}
\usepackage{amsfonts,amssymb, amscd}
\usepackage{amsmath}
\usepackage{lmodern}
\usepackage{mathrsfs}
\usepackage{amsthm}
\usepackage{qsymbols}
\usepackage{mathtools}
\mathtoolsset{showonlyrefs}
\usepackage{dsfont}
\usepackage{graphicx}
\usepackage{latexsym}
\usepackage{chngcntr}
\usepackage[noadjust]{cite}
\usepackage{paralist}
\usepackage[parfill]{parskip}
\usepackage{bm}
\usepackage{tikz-cd}
\usetikzlibrary{calc}
\usepackage{nicefrac}
\usepackage{float}
\usepackage{csquotes}
\usepackage{esint}
\usepackage{multirow}
\usepackage{todonotes}
\usepackage[shortlabels]{enumitem}
\setlist[enumerate,1]{label = \normalfont(\arabic*), ref = (\arabic*)}
\usepackage{cancel}
\usepackage{comment}
\usepackage{euflag}

\newtheorem{theorem}{Theorem}[section]
\newtheorem{lemma}[theorem]{Lemma}
\newtheorem{proposition}[theorem]{Proposition}
\newtheorem{corollary}[theorem]{Corollary}
\newtheorem{taggedtheoremx}{Assumption}
\newenvironment{assumption}[1]
{\renewcommand\thetaggedtheoremx{#1}\taggedtheoremx}
{\endtaggedtheoremx}
\theoremstyle{definition}
\newtheorem{definition}[theorem]{Definition}
\newtheorem{remark}[theorem]{Remark}
\newtheorem{example}[theorem]{Example}

\usepackage[backref=page]{hyperref}
\hypersetup{
	plainpages=false,
	pdfpagelabels,
	colorlinks,
	linkcolor={cyan!90!black},
	citecolor={magenta},
	urlcolor={green!40!black},
	bookmarksdepth=2,
} %

\newcommand{\R}{\mathbb{R}}

\renewcommand{\L}{\mathrm{L}}

\newcommand{\VMO}{\mathrm{VMO}}
\newcommand{\BUC}{\mathrm{BUC}}
\newcommand{\X}{\mathrm{X}}

\newcommand{\Z}{\mathrm{Z}}

\let\SS\S
\renewcommand{\S}{\mathcal{S}}

\newcommand{\Ext}{\mathcal{E}}
\newcommand{\Sol}{\mathcal{R}}

\newcommand{\Cont}{\mathrm{C}}

\newcommand{\e}{\mathrm{e}}
\renewcommand{\d}{\,\mathrm{d}}

\newcommand{\cc}{\mathrm{c}}

\newcommand{\eps}{\varepsilon}
\newcommand{\B}{\mathrm{B}}

\renewcommand{\H}{\mathrm{H}}

\newcommand{\ind}{\mathbf{1}}

\DeclareMathOperator{\dist}{d}

\DeclareMathOperator{\Div}{div}
\DeclareMathOperator{\Ran}{ran}
\DeclareMathOperator{\conv}{conv}
\DeclareMathOperator*{\esssup}{ess\,sup}

\newcommand{\loc}{\mathrm{loc}}

\newcommand{\Id}{\mathrm{Id}}

\def\YYint#1#2#3{{\setbox0=\hbox{$#1{#2#3}{\iint}$}
		\vcenter{\hbox{$#2#3$}}\kern-.51\wd0}}

\makeatletter
\@namedef{subjclassname@2020}{%
	\textup{2020} Mathematics Subject Classification}
\makeatother

\title[Weak solutions to quasilinear PDEs with critical data]{Existence and uniqueness of weak solutions to quasilinear PDEs with critical data}

\author{Pascal Auscher}
\address{Université Paris-Saclay
\\ France-Australia Mathematical Sciences and Interactions ANU - CNRS International Research Laboratory, Canberra, ACT 2601, Australia.}
\email{pascal.auscher@universite-paris-saclay.fr}
\author{Sebastian Bechtel}
\address{Université Paris-Saclay, CNRS \\ Laboratoire de Mathématiques d’Orsay \\ 91405 Orsay \\ France}
\email{sebastian.bechtel@universite-paris-saclay.fr}
\date{\today}
\dedicatory{}
\thanks{This project has received funding from the European Union’s Horizon 2020 research and innovation programme under the Marie Skłodowska-Curie grant agreement No 101034255 \euflag{}.
The second-named author thanks the CNRS IRL FAMSI for support.
A CC-BY 4.0 \url{https://creativecommons.org/licenses/by/4.0/} public copyright license has been applied by the authors to the present document and will be applied to all subsequent versions up to the Author Accepted Manuscript arising from this submission.}
\begin{document}
	\begin{abstract}

		We establish existence and uniqueness of global, bounded weak solutions to quasilinear PDEs with bounded, uniformly continuous initial data and investigate their properties. Moreover, we establish existence of bounded weak solutions when the initial data is merely bounded.

	\end{abstract}
	\maketitle

	\allowdisplaybreaks

	\setcounter{tocdepth}{1}
	\tableofcontents

	\section{Introduction}
	\label{sec:intro}

	Fix the dimension $n \geq 1$. We consider the quasilinear problem
	\begin{align}
		\label{eq:ql}
		\tag{QL}
		\partial_t u - \Div(a(t,x,u)\nabla u) = 0, \quad u(0) = u_0.
	\end{align}

	Here, the coefficient function $a \colon [0,\infty) \times \R^n \times O \to \R^{n\times n}$ with $O \subseteq \R$ open satisfies Assumption~\ref{ass:a} below, and $u_0 \in \BUC(\R^n)$ with the closure of its range $\overline{\Ran(u_0)}$ contained in $O$.

	\begin{assumption}{A}
		\label{ass:a}
		The non-linearity $a$ satisfies the following:
		\begin{enumerate}
			\item\label{it:a1} (Local ellipticity) For all compact $K \subseteq O$ there exists $\lambda(K) \in (0,\infty)$ such that
			\begin{align}
				\label{eq:elliptic}
				a(t,x,y) \xi \cdot \xi \geq \lambda(K) |\xi|^2, \quad t \in [0,\infty), x\in \R^n, y\in K, \xi \in \R^n.
			\end{align}
			\item\label{it:a2} (Local Lipschitz condition) For all $T \in [0,\infty)$ and all compact $K \subseteq O$ there exists a finite constant $C_L(T,K)$ non-decreasing in $T$ such that, for all $t\in [0,T]$, $x \in \R^n$ and $y,y' \in K$, there holds
			\begin{align}
				\label{eq:loc_lip}
				|a(t,x,y) - a(t,x,y')| \leq C_L(T,K) |y-y'|.
			\end{align}
			\item\label{it:a3} (Bounded equilibrium) For all $T \in [0,\infty)$ there exists a finite constant $C_E(T)$ non-decreasing in $T$ such that, for all $t\in [0,T]$ and $x\in \R^n$, there holds
			\begin{align}
				\label{eq:equilibrium}
				|a(t,x,0)| \leq C_E(T).
			\end{align}
			\item\label{it:a4} (Local uniform continuity) For all $T \in (0,\infty)$ and all compact $K \subseteq O$, the restriction of $a$ to $[0,T] \times \R^n \times K$ is uniformly continuous.
		\end{enumerate}
	\end{assumption}

	The following example shows that $a(t,x,y)$ can be unbounded and degenerate in $y$. Nevertheless, our quasilinear problem has the nature of a uniformly parabolic problem thanks to the choice of the initial datum $u_0$.

	\begin{example}[Unbounded and degenerate $a$]
		For $m \in (0,\infty)$,
		consider the quasilinear problem
		\begin{align}
			\partial_t u - \Div(u^m \nabla u) = 0, \quad u(0) = u_0.
		\end{align}
		It corresponds to $a(t,x,y) \coloneqq y^m \cdot \Id$ and satisfies Assumption~\ref{ass:a} when $O \coloneqq (0, \infty)$. Then, every $u_0 \in \BUC(\R^n)$ with $\inf_{x\in \R^n} u_0(x) > 0$ is an admissible initial condition.
	\end{example}

	Our goal is to show existence of a unique global bounded weak solution $u \colon (0, \infty) \times \R^n \to O$ to~\eqref{eq:ql} and to investigate its properties.
	Our definition of a bounded weak solution to~\eqref{eq:ql} is the following.

	\begin{definition}[Bounded weak solutions to~\eqref{eq:ql}]
		\label{def:weak_ql}
		Let $O \subseteq \R$ open and $T \in (0,\infty]$.
		Given $a \colon (0,T) \times \R^n \times O \to \R^{n\times n}$ measurable and locally bounded, and $u_0 \colon \R^n \to \R$ measurable and locally bounded, call $u\colon (0,T) \times \R^n \to O$ a \emph{bounded weak solution} to~\eqref{eq:ql} if $u$ is essentially bounded with $\overline{\vphantom{t}\Ran}(u) \subseteq O$, $\nabla u \in \L^2_\loc((0,T) \times \R^n)$, and $u$ satisfies the equation weakly:
		\begin{align}
			\iint_{(0,T) \times \R^n} u (-\partial_t \phi) + a(t,x,u)\nabla u \cdot \nabla \phi \d x \d t = 0, \quad \phi \in \Cont_\cc^\infty((0,T) \times \R^n),
		\end{align}
		and the distributional limit:
		\begin{align}
			\label{eq:weak_ql_def_limit}
			\langle u(t), \psi \rangle \to \langle u_0, \psi \rangle \;\;\text{as}\;\; t\to 0, \quad \psi \in \Cont_\cc^\infty(\R^n).
		\end{align}
	\end{definition}
	\begin{remark}[Trace]
		The limit~\eqref{eq:weak_ql_def_limit} always makes sense from the other assumptions, see Proposition~\ref{prop:ql_apriori}\ref{it:ql1}.
	\end{remark}

	The following is our main result.

	\begin{theorem}[Well-posedness and regularity of~\eqref{eq:ql}]
		\label{thm:main}
		Let $O \subseteq \R$ be open and suppose that $a \colon [0,\infty) \times \R^n \times O \to \R^{n\times n}$ satisfies Assumption~\ref{ass:a}.
		Then, for every $u_0 \in \BUC(\R^n)$ with $\overline{\Ran(u_0)} \subseteq O$
		the following hold:
		\begin{itemize}
			\item (Existence and Uniqueness) There exists a unique, global, bounded weak solution $u \colon (0,\infty) \times \R^n \to O$ to~\eqref{eq:ql}.
			\item (Regularity) The solution $u$ is uniformly continuous on $[0,\infty) \times \R^n$,
			satisfies the Carleson measure estimate
			\begin{align}
				\sup_{x\in \R^n} \sup_{t \in (0,\infty)}  \bigg( \int_0^t \fint_{\B(x,\sqrt{t})} |\nabla u(s,y)|^2 \d y \d s \bigg)^\frac{1}{2} \leq C(\overline{\Ran(u_0)}) \| u_0 \|_\infty,
			\end{align}
			and for all $T\in (0,\infty)$ and $q\in (1,\infty)$ we have
			\begin{align}
				\sup_{x\in \R^n} \sup_{t \in (0,T)}  \bigg( \fint_{t/2}^t \fint_{\B(x, \sqrt{t})} |s^{\frac{1}{2}} \nabla u(s,y)|^q \d y \d s \bigg)^\frac{1}{q} < \infty.
			\end{align}
			\item (Stability) We have $\overline{\Ran(u)} = \overline{\Ran(u_0)}$ and $u(t) \to u_0$ in $\L^\infty(\R^n)$ as $t \to 0$. Moreover, if $\lim_{|x| \to \infty} u_0(x) = c$ exists, then $u(t) \to c$ in $\L^\infty(\R^n)$ as $t \to \infty$.
		\end{itemize}
	\end{theorem}

	Let us highlight a few points in this result.
	For sufficiently regular $a$ and $u_0$, wellposedness and regularity of bounded weak solutions to~\eqref{eq:ql} were established by Ladyzhenskaya, Solonnikov and Ural'tseva in~\cite[Ch.\,V,~Thm.\,8.1]{LSU}.
	Here, we push things to a scaling-critical minimal regularity setting, which has not been addressed before.
	Our approach bears on the corresponding autonomous and non-autonomous theory of bounded weak solutions to parabolic linear problems $$\partial_t u -\Div(A\nabla u) = \Div(F), \quad u(0) = u_0,$$
	with $A$ uniformly elliptic (and potentially uniformly continuous) and $u_0$ bounded (and potentially uniformly continuous).
	Boundedness is a scale invariant condition under parabolic scaling on solutions, so taking $u_0$ bounded is natural. Thus, we need to work with a class of source terms $F$ respecting the invariance likewise.
	The non-autonomous theory for the homogeneous equation with measurable $A$ and $u_0$ bounded is quite complete (existence, uniqueness, regularity).
	The scaling suggests to control $\sqrt{t} \nabla u$ uniformly, but this is too demanding in our low regularity regime. Instead, we use solution spaces in which we replace pointwise control by $\L^q$ control in average following scale considerations, and we use the same spaces for source terms. Such spaces, belonging to the family of weighted $\Z$-spaces, were already used in our prior work~\cite{AB} on reaction--diffusion equations to derive new bounds for autonomous linear problems. These bounds give us the possibility to setup a fixed-point argument for existence if we take $q$ large enough. Bounded weak solutions are shown to be a priori uniformly continuous, which is the key to establish that they are global and unique.

	In contrast to large parts of the literature, we do not work in a bounded domain of the Euclidean space.
	The unboundedness of $\R^n$ is another key difficulty: it rules out the usage of classical compactness arguments to obtain existence, and the mass of $u_0$ is in general infinite, so that uniqueness of solutions cannot be ensured by $\L^1$-techniques.

	Having Theorem~\ref{thm:main} in hand enables us to drop various conditions on $a$ and $u_0$ and still obtain existence of bounded weak solutions, this time, by compactness arguments. For example, if we drop uniform continuity of $u_0$, we obtain the following result, improving on~\cite[Ch.\,V, Rem.~8.1]{LSU}.

	\begin{corollary}[Existence to~\eqref{eq:ql} with bounded data]
		\label{cor:Loo}
		Let $O \subseteq \R$ be open and suppose that $a \colon [0,\infty) \times \R^n \times O \to \R^{n\times n}$ satisfies Assumption~\ref{ass:a}.
		Then, for every $u_0 \in \L^\infty(\R^n)$ with $\overline{\vphantom{t}\Ran}(u_0) \subseteq O$, there exists a global, bounded weak solution $u \colon (0,\infty) \times \R^n \to O$ to~\eqref{eq:ql}.
	\end{corollary}

	\begin{remark}[Discussion of the corollary]
		Compared to Theorem~\ref{thm:main}, we loose uniqueness, $\Z$-space regularity of the gradient and stability as $t\to 0$.
		Stability as $t \to \infty$ and the Carleson measure estimate for the gradient are still valid, and so is the range condition for bounded data discussed in Remark~\ref{rem:ql_apriori}. Existence can be proved under weaker assumptions on $a$ as well. This can be done by regularization of $a$ and a compactness argument as for Corollary~\ref{cor:Loo}. Besides Assumption~\ref{ass:a}\ref{it:a1}, it suffices that for almost every $(t,x) \in [0,\infty) \times \R^n$, $a(t,x,\cdot)$ is continuous on $O$, and that $a \in \L^\infty([0,T] \times \R^n \times K)$ for all $T \in (0,\infty)$ and compact $K \subseteq O$.
	\end{remark}

	\begin{remark}[Hölder regularity]
		Let $u$ be any bounded weak solution as constructed in Corollary~\ref{cor:Loo}.
		By Assumption~\ref{ass:a}\ref{it:a1}, $A(t,x) \coloneqq a(t,x,u(t,x))$ is uniformly elliptic on $(0,\infty) \times \R^n$. By Nash's local regularity~\cite{Nash}, $u \in\Cont^\alpha([\delta', \infty) \times \R^n)$ for some $\alpha \in (0,1)$ and all $\delta' \in (0,\infty)$. In particular, $u \in \BUC([\delta', \infty) \times \R^n)$, and therefore $A$ is in addition uniformly continuous on $[\delta',T] \times \R^n$ for all $T \in (\delta', \infty)$ by Assumption~\ref{ass:a}\ref{it:a4}. So, classical parabolic regularity theory (see~\cite[Ch.\,IV, \SS9]{LSU}) provides us with $u \in\Cont^\alpha([\delta, T] \times \R^n)$ for all $\alpha \in (0,1)$, $T \in (\delta',\infty)$ and $\delta \in (\delta',T)$.
		This remark applies likewise to the bounded weak solution constructed in Theorem~\ref{thm:main}.
	\end{remark}

	\subsection*{Open problems}

	We conclude with some open problems:

	\begin{itemize}
		\item (Systems) Our Proposition~\ref{prop:cp_wp} remains valid for systems, hence it would be interesting to look at quasilinear systems with our approach.
		\item (Uniqueness) For irregular initial data such as $u_0 \in \L^\infty(\R^n)$, it would be interesting to investigate suitable conditions ensuring uniqueness of bounded weak solutions. The issue is to have sufficient a priori regularity near the initial time.
		\item (Existence) It would be interesting to investigate whether the existence result in Corollary~\ref{cor:Loo} can be extended to more irregular data such as $u_0 \in \VMO(\R^n)$.
	\end{itemize}

	\subsection*{Notation}
	\label{subsec:not}

	We use generic constants $C \in (0,\infty)$ as well as the Vinogradov notation $\lesssim$, $\gtrsim$ and $\approx$. Generic constants are allowed to depend on the dimension $n$ without further mentioning.
	To emphasize dependence on a parameter $P$, we may write $C=C(P)$.
	For $p \in (1,n+2)$ define the upper parabolic Sobolev conjugate $p^* \in (p,\infty)$ by
	$$\frac{1}{p^*} = \frac{1}{p} - \frac{1}{n+2},$$
	and for $q \in (1, \infty)$ define the lower parabolic Sobolev conjugate $q_* \in (0,q)$ by
	$$\frac{1}{q_*} = \frac{1}{q} + \frac{1}{n+2}.$$
	For $x \in \R^n$ and $r \in (0,\infty)$, we write $\B(x,r)$ for the open Euclidean ball of radius $r$ around~$x$.
	We write $\conv(A)$ for the convex hull of a set $A \subseteq \R$.
	The space $\BUC(\R^n)$ consists of all bounded and uniformly continuous real-valued functions on $\R^n$ and is equipped with the uniform norm $\| \cdot \|_\infty$.
	In particular, $\BUC(\R^n)$ is a closed subspace of $\L^\infty(\R^n)$.
	We use the notation $\| \cdot \|_\infty$ also for the essential supremum norm on $\L^\infty(\R^n)$.
	For a measurable, real-valued function $f$, let $\overline{\vphantom{t}\Ran}(f)$ denote the essential range of $f$. The essential range is always closed. When $f$ is continuous, then $\overline{\vphantom{t}\Ran}(f) = \overline{\Ran(f)}$.
	If $X$ is any function space and $F$ is an $\R^n$-valued vector field, we still write $F \in X$ instead of $F \in X^n$ to simplify the notation. For a set $A$ of positive measure, we write $\fint_A$ for the average integral $\tfrac{1}{|A|} \int_A$.

	\subsection*{Acknowledgments}
	\label{subsec:ack}

	The second-named author thanks the Australian National University for its hospitality during a research visit in February 2026. Both authors are grateful to Pierre Portal for valuable discussions.

	\section{Preliminaries on function spaces}
	\label{sec:preliminaries}

	\begin{definition}[$\Z$-spaces]
		\label{def:Z}
		Let $q \in (1,\infty]$ and $T \in (0,\infty]$. The space $\Z^{\infty,q}_{-\nicefrac{1}{2}}(T)$ consists of all measurable functions $f \colon (0,T) \times \R^n \to \R$ such that
		\begin{align}
			\| f \|_{\Z^{\infty,q}_{-\frac{1}{2}}(T)} \coloneqq \sup_{x\in \R^n} \sup_{t \in (0,T)} \bigg( \fint_{t/2}^t \fint_{\B(x, \sqrt{t})} |s^\frac{1}{2} f(s,y)|^q \d y \d s \bigg)^\frac{1}{q} < \infty,
		\end{align}
		with the obvious modification if $q=\infty$.
	\end{definition}

	For $\beta \in \R$ and $T\in (0,\infty]$, define the space $\L^\infty_\beta(T)$ as the collection of all measurable functions $f \colon (0,T) \times \R^n \to \R$ such that
	\begin{align}
		\| f \|_{\L^\infty_\beta(T)} \coloneqq \esssup_{(t,x) \in (0,T) \times \R^n} |t^{-\beta} f(t,x)| < \infty.
	\end{align}
	Clearly, $\Z^{\infty,\infty}_{-\nicefrac{1}{2}}(T) = \L^\infty_{-\nicefrac{1}{2}}(T)$.
	Note that, by Jensen's inequality, for every $q_0, q_1 \in (1,\infty]$ with $q_0 \geq q_1$ there holds the contractive inclusion $\Z^{\infty,q_0}_{-\nicefrac{1}{2}}(T) \subseteq \Z^{\infty,q_1}_{-\nicefrac{1}{2}}(T)$. We call this inclusion the \emph{nesting property} of weighted $\Z$-spaces. In particular, $\L^\infty_{-\nicefrac{1}{2}}(T)$ is contained in $\Z^{\infty,q}_{-\nicefrac{1}{2}}(T)$ for any $q \in (1,\infty]$.
	For brevity, we put $\L^\infty(T) \coloneqq \L^\infty_0(T)$ and we simply write $\L^\infty_\beta$ when $T = \infty$.

	\begin{remark}
		\label{rem:Z_not_L2}
		Let $q \in (1,\infty]$ and $T \in (0,\infty]$. The space $\Z^{\infty,q}_{-\nicefrac{1}{2}}(T)$ is not contained in $\L^2(0,T; \L^2_\loc(\R^n))$. Indeed, this is clear for $q < 2$, and the function $f(s,y) \coloneqq s^{-\nicefrac{1}{2}} \ind_B(y)$ serves as a counter-example when $q \geq 2$, where $B$ denotes the unit ball.
	\end{remark}

	\section{A priori properties for non-autonomous linear PDEs}
	\label{sec:non-autonomous}

	Fix $T \in (0, \infty]$ and $A \colon (0,T) \times \R^n \to \R^{n\times n}$ measurable, bounded and uniformly elliptic throughout the section.
	Here and in the sequel, we say that $A$ is \emph{uniformly elliptic} if there exists $\lambda \in (0,\infty)$ such that
	\begin{align}
		A(t,x)\xi \cdot \xi \geq \lambda |\xi|^2, \quad t \in (0,T), x\in \R^n, \xi \in \R^n.
	\end{align}
	We call $\| A \|_{\L^\infty(T)}$ and the largest possible value of $\lambda$ the \emph{coefficient bounds} of $A$.
	The goal of this section is to establish a priori properties for solutions to the \emph{non-autonomous} homogeneous problem
	\begin{align}
		\label{eq:CP_NAT}
		\tag{NA}
		\partial_t u -\Div(A\nabla u) = 0.
	\end{align}

	The following is our definition of a weak solution. As a preparation for Section~\ref{sec:homogeneous}, we also include a non-trivial source term.

	\begin{definition}[Weak solutions]
		\label{def:weak_solution}
		Given $F \in \L^2_\loc((0,T) \times \R^n)$, call $u\colon (0,T) \times \R^n \to \R$ a \emph{weak solution} to
		\begin{align}
			\label{eq:def_weak_eqn}
			\partial_t u -\Div(A\nabla u) = \Div(F),
		\end{align}
		if $u,\nabla u \in \L^2_\loc((0,T) \times \R^n)$ and $u$ satisfies the equation weakly:
			\begin{align}
				\iint_{(0,T) \times \R^n} u (-\partial_t \phi) + (A\nabla u + F) \cdot \nabla \phi \d x \d t = 0, \quad \phi \in \Cont_\cc^\infty((0,T) \times \R^n).
			\end{align}
	\end{definition}

	The following lemma is a consequence of Lions' embedding, see~\cite{Lions57}, assuming square integrability of the gradient up to the initial time.

	\begin{lemma}[Regularity and integral identity]
		\label{lem:Lions_cont}
		For $F \in \L^2(0,T; \L^2_\loc(\R^n))$, let $u \colon (0,T) \times \R^n \to \R$ be a weak solution to $$\partial_t u -\Div(A\nabla u) = \Div(F).$$
		Let $T' \in (0,T] \cap (0,\infty)$ and assume $\nabla u \in \L^2(0,T'; \L^2_\loc(\R^n))$.
		Then, $u \in \Cont([0,T']; \L^2_\loc(\R^n))$ and for all $\phi \in \Cont^\infty_\cc([0,T'] \times \R^n)$ there holds
		\begin{align}
			\langle u(0), \phi(0) \rangle = \langle u(T'), \phi(T') \rangle + \iint_{(0,T') \times \R^n} u(-\partial_t \phi) + (A\nabla u + F)\cdot \nabla \phi \d x \d t.
		\end{align}
	\end{lemma}

	\begin{remark}[Initial conditions]
		\label{rem:initial_conditions}
		For every $u_0 \in \L^2_\loc(\R^n)$, Lemma~\ref{lem:Lions_cont} allows us to complement~\eqref{eq:def_weak_eqn} by an initial condition $u(0) = u_0$ provided we assume square integrability of $\nabla u$ up to the initial time.
		This condition is satisfied for bounded weak solutions when $F = 0$.
	\end{remark}

	The following proposition provides us with a priori properties for bounded weak solutions to the homogeneous Cauchy problem~\eqref{eq:CP_NAT}.

	\begin{proposition}[A priori properties]
		\label{prop:apriori}
		Let $u \colon (0,T) \times \R^n \to \R$ be a bounded weak solution to~\eqref{eq:CP_NAT}.

		The following hold:
		\begin{enumerate}[label=(\alph*)]
			\item\label{it:Carleson} There exists a finite constant $C$ depending on the coefficient bounds of $A$, such that we have the Carleson measure estimate
			\begin{align}
				\sup_{x\in \R^n} \sup_{t \in (0,T]} \bigg( \int_0^t \fint_{\B(x,\sqrt{t})} |\nabla u(s,y)|^2 \d y \d s \bigg)^\frac{1}{2} \leq C \| u \|_{\L^\infty(T)}.
			\end{align}
			\item\label{it:Loo} There exists $u(0) \in \L^\infty(\R^n)$ such that $u(t) \to u(0)$ in $\L^2_\loc(\R^n)$ as $t \to 0$.
			\item\label{it:uniqueness} If $u(0) = 0$, then $u = 0$.
			\item\label{it:range}\label{it:mp} There holds $\conv(\overline{\vphantom{t}\Ran}(u)) = \conv(\overline{\vphantom{t}\Ran}(u(0)))$. In particular,
			\begin{align}
				\| u \|_{\L^\infty(T)} = \| u(0) \|_\infty.
			\end{align}
			\item\label{it:buc} If $u(0) \in \BUC(\R^n)$, then $u$ is uniformly continuous on $[0,T) \times \R^n$. The modulus of continuity of $u$ depends on the modulus of continuity of $u(0)$, $\| u(0) \|_\infty$ and the coefficient bounds of $A$. In particular, if $T < \infty$, we have $u(T) \in \BUC(\R^n)$.
			\item\label{it:buc_grad} If $A$ is uniformly continuous, then for all $T' \in (0,T] \cap (0,\infty)$ and $q\in (1,\infty)$ we have $\nabla u \in \Z^{\infty,q}_{-\nicefrac{1}{2}}(T')$.
			Moreover, there exists a finite constant $C(T')$ non-decreasing in $T'$ and depending on $q$, and the modulus of continuity and coefficient bounds of $A$, such that
			\begin{align}
				\| \nabla u \|_{\Z^{\infty,q}_{-\frac{1}{2}}(T')} \leq C(T') \| u(0) \|_\infty.
			\end{align}
		\end{enumerate}
	\end{proposition}

	\begin{proof}
		Part~\ref{it:Carleson} is a special case of~\cite[Thm.~7.4]{AMP19}, taking its localization to finite time intervals (explained in~\cite[Sec.~10]{AMP19}) into account. We also provide an elementary proof of this fact in Proposition~\ref{prop:carleson}.

		For part~\ref{it:Loo}, note that $\nabla u \in \L^2(0,T'; \L^2_\loc(\R^n))$ for all $T' \in (0,T] \cap (0, \infty)$ by~\ref{it:Carleson}. Hence, the existence of a trace $u(0) \in \L^\infty(\R^n)$ follows from Lemma~\ref{lem:Lions_cont} and extraction of a subsequence $(u(t_k))_k$ converging almost everywhere to $u(0)$.

		Next, part~\ref{it:uniqueness} follows from~\cite[Thm.~2]{Aronson} using that the class $\mathcal{E}^2((0,T') \times \R^n)$ contains $\L^\infty(T')$ for all $T' \in (0,T] \cap (0,\infty)$.

		As a preparation to obtain~\ref{it:range}, we recall the generalized fundamental solution $\Gamma(t,x,s,y)$ to~\eqref{eq:CP_NAT} constructed by Aronson in~\cite{Aronson}. It satisfies for some $c_0,c_1,C_0,C_1 \in (0,\infty)$ the Gaussian lower and upper bounds
		\begin{align}
			\label{eq:Gaussian}
			C_0 (t-s)^{-\frac{n}{2}} \e^{-c_0\frac{|x-y|^2}{t-s}} \leq \Gamma(t,x,s,y) \leq C_1 (t-s)^{-\frac{n}{2}} \e^{-c_1\frac{|x-y|^2}{t-s}}.
		\end{align}

		Now, we prove~\ref{it:range}. By~\ref{it:Loo}, $\overline{\vphantom{t}\Ran}(u(0))$ is compact.
		Since the convex hull of any non-empty compact set of real numbers is a compact interval, we can write $\conv(\overline{\vphantom{t}\Ran}(u(0))) = [a,b]$ for some $a, b \in \R$ with $a \leq b$.
		First, to see $\conv(\overline{\vphantom{t}\Ran}(u)) \subseteq \conv(\overline{\vphantom{t}\Ran}(u(0)))$, we show $a \leq u \leq b$ on $(0,T) \times \R^n$ almost everywhere.
		To this end, we use the representation
		\begin{align}
			u(t,x) = \int_{\R^n} \Gamma(t,x,0,y) u(0)(y) \d y,
		\end{align}
		which follows from~\ref{it:uniqueness} as its right-hand side is also a bounded weak solution to~\eqref{eq:CP_NAT} with the same initial condition as $u$. Since we have $a\le u_{0}\le b$ almost everywhere, $\Gamma\ge 0$ and $\int_{\R^n} \Gamma(t,x,0,y) \d y = 1$, we obtain $a \leq u \leq b$ on $(0,T) \times \R^n$ almost everywhere.
				Second, the converse inclusion $\conv(\overline{\vphantom{t}\Ran}(u(0))) \subseteq \conv(\overline{\vphantom{t}\Ran}(u))$ follows from the convergence of $u(t)$ to $u(0)$ in $\L^2_\loc(\R^n)$ proved in~\ref{it:Loo}.

		We proceed with~\ref{it:buc}. For weak solutions to parabolic equations with uniformly elliptic coefficients on a bounded parabolic cylinder $(0,T') \times \Omega$, $T' \in (0,T] \cap (0,\infty)$, the result is stated in~\cite[Thm.~D]{Aronson}, relying on the boundary Harnack principle from~\cite{Trudinger}. In fact, the result extends to parabolic equations in $(0,T)\times \R^n$. Indeed, for every $x\in \R^n$, we can apply~\cite[Thm.~4.2]{Trudinger} locally on $(0,T') \times \B(x,1)$ with parabolic cylinders staying away from the lateral boundary $(0,T')\times \partial\B(x,1)$. Since the bound obtained in~\cite[Thm.~4.2]{Trudinger} is independent of $T'$ and $\Omega$, we can control the oscillation uniformly over $(0,T) \times \R^n$. Moreover, this extends to uniform continuity on $[0,T] \times \R^n$ when $T$ is finite, so $u(T) \in \BUC(\R^n)$.

		Finally, we show~\ref{it:buc_grad}.
		By the nesting property of weighted $\Z$-spaces, we can assume $q \geq 2$.
		For $x\in \R^n$ and $t \in (0,T')$, we have to show
		\begin{align}
			\bigg( \fint_{t/2}^t \fint_{\B(x,\sqrt{t})} |s^{\frac{1}{2}} \nabla u|^q \d y \d s \bigg)^\frac{1}{q} \lesssim \| u(0) \|_\infty.
		\end{align}
		By a translation and scaling argument, it suffices to show
		\begin{align}
			\label{eq:buc_grad_claim_for_v}
			\bigg( \int_{1/2}^1 \int_{\B(0,1)} |\nabla v|^q \d y \d s \bigg)^\frac{1}{q} \lesssim \| v(0) \|_\infty,
		\end{align}
		where $v \colon (0,1) \times \R^n \to \R$ is a bounded weak solution to
		\begin{align}
			\label{eq:buc_grad_aux}
			\partial_t v - \Div(\hat A \nabla v) = 0, \quad v(0) = v_0.
		\end{align}
		Here, $\hat A \colon (0,1) \times \R^n \to \R^{n\times n}$ is bounded, uniformly elliptic and uniformly continuous, and $v(0) \in \L^\infty(\R^n)$. The modulus of continuity of $\hat A$ depends on the modulus of continuity of $A$ as well as of $T'$ in a non-decreasing way.
		To prove~\eqref{eq:buc_grad_claim_for_v}, we shall show for all $r \in (1^*, \infty)$ and $0 < c' < c < 1 < d < d' < \infty$ the inequality
		\begin{align}
			\label{eq:buc_grad_iteration}
			\bigg( \int_c^1 \int_{\B(0,d)} |\nabla v|^r \d y \d s \bigg)^\frac{1}{r} \lesssim \bigg( \int_{c'}^1 \int_{\B(0,d')} |\nabla v|^{r_*} \d y \d s \bigg)^\frac{1}{r_*} + \| v \|_{\L^\infty(1)},
		\end{align}
		where the implicit constant depends on the modulus of continuity and the coefficient bounds of $\hat A$, $r$, $c$, $c'$, $d$ and $d'$. Iterating this inequality from $r_*$ instead of $r$ and so on allows us to reach an exponent $p\in (1,2]$ instead of $r_*$ on the right-hand side of~\eqref{eq:buc_grad_iteration} in finitely many steps. Then, using Caccioppoli's inequality (see~\cite[Prop.~3.2]{AMP19}) and part~\ref{it:mp} applied to $v$, we have
		\begin{align}
			\bigg( \int_{c'}^1 \int_{\B(0,d')} |\nabla v|^p \d y \d s \bigg)^\frac{1}{p} &\lesssim \bigg( \int_{c'}^1 \int_{\B(0,d')} |\nabla v|^2 \d y \d s \bigg)^\frac{1}{2} \\
			 &\lesssim \bigg( \int_{c'/2}^1 \int_{\B(0,2d')} |v|^2 \d y \d s \bigg)^\frac{1}{2}
			 \lesssim \| v \|_{\L^\infty(1)} \leq \| v(0) \|_\infty.
		\end{align}
		It remains to show~\eqref{eq:buc_grad_iteration}.
		Let $\psi \in \Cont_\cc^\infty((c',1] \times \B(0,d'))$ satisfying $\psi = 1$ on $[c,1] \times \overline{\B(0,d)}$. Define $w \coloneqq v\psi$, $f \coloneqq v \partial_t \psi - \hat A \nabla v \cdot \nabla \psi$ and $\boldsymbol{g} \coloneqq v \hat A \nabla \psi$. Then, $w \colon (0,1) \times \B(0,d') \to \R$ is a bounded weak solution (adapting the definition to $(0,1)\times \B(0,d')$) to
		\begin{align}
			\partial_t w - \Div(\hat A \nabla w) = f + \Div(\boldsymbol{g})
		\end{align}
		satisfying $w = 0$ on $\bigl((0,1) \times \partial \B(0,d') \bigr) \cup \bigl( \{0 \} \times \B(0,d') \bigr)$.
		Since $\hat A$ is uniformly continuous, we may follow the proof of~\cite[Prop.~4.1]{DEK} with $\mathcal{P} = \partial_t - \Div(\hat A \nabla)$ (hence, their coefficients $b$, $c$ and $d$ are zero) to obtain
		\begin{align}
			\bigg( \int_0^1 \int_{\B(0,d')} |\nabla w|^r \d y \d s \bigg)^\frac{1}{r} \lesssim \bigg( \int_0^1 \int_{\B(0,d')} |f|^{r_*} \d y \d s \bigg)^\frac{1}{r_*} + \| \boldsymbol{g} \|_{\L^\infty((0,1) \times \B(0, d'))},
		\end{align}
		with implicit constant depending on the modulus of continuity and coefficient bounds of $\hat A$ and $r$.
		Now, using the support of $\psi$ and the definition of $f$ and $\boldsymbol{g}$, we deduce~\eqref{eq:buc_grad_iteration}.
	\end{proof}

	\begin{remark}[Carleson measure estimate]
		The Carleson measure estimate in~\ref{it:Carleson} is scale invariant.
		A particular consequence of it is $\nabla u \in \L^2(0,T;\L^2_\loc(\R^n))$, compare with Remark~\ref{rem:initial_conditions}.
		Another different argument uses the representation by the fundamental solution following from~\ref{it:Loo}.
	\end{remark}

	\begin{remark}[Range property]
		\label{rem:conv_hull}
		If $u(0) \in \BUC(\R^n)$, then by part~\ref{it:buc} we know that both $\overline{\vphantom{t}\Ran}(u) = \overline{\Ran(u)}$ and $\overline{\vphantom{t}\Ran}(u(0)) = \overline{\Ran(u(0))}$ are connected. Therefore, the convex hull in part~\ref{it:range} is superfluous, and the range property reads $\overline{\Ran(u)} = \overline{\Ran(u(0))}$.
	\end{remark}

	\section{Bounded weak solutions to autonomous linear PDEs}
	\label{sec:homogeneous}

	Fix $A_0 \colon \R^n \to \R^{n\times n}$ bounded, uniformly elliptic and uniformly continuous throughout this section. For $u_0 \in \BUC(\R^n)$ and $F\in \Z^{\infty,q}_{-\nicefrac{1}{2}}(T)$, $T \in (0, \infty)$, we investigate the \emph{autonomous} linear problem
	\begin{align}
		\label{eq:CP}
		\tag{A}
		\partial_t u -\Div(A_0\nabla u) = \Div(F), \quad u(0) = u_0.
	\end{align}

	Of course, the definitions and results of Section~\ref{sec:non-autonomous} apply in particular to the autonomous case.
	However, since $F$ is not square integrable up to the initial time here (Remark~\ref{rem:Z_not_L2}), things get slightly different, especially with regard to the initial condition.

	\begin{proposition}[Well-posedness of~\eqref{eq:CP}]
		\label{prop:cp_wp}
		Fix $T \in (0, \infty)$ and $q\in (n+2, \infty)$. Then, for all $u_0 \in \BUC(\R^n)$ and $F \in \Z^{\infty,q}_{-\nicefrac{1}{2}}(T)$ there exists a unique, bounded weak solution $u \colon (0,T) \times \R^n \to \R$ to~\eqref{eq:CP} with $u(t) \to u_0$ in $\mathcal{D}'(\R^n)$. Moreover, $\nabla u \in \Z^{\infty,q}_{-\nicefrac{1}{2}}(T)$ and there is a constant $C(T)$, non-decreasing in $T$ and depending on the modulus of continuity and coefficient bounds of $A_0$ and $q$, such that
		\begin{align}
			\label{eq:wp_fp_est1}
			\| u \|_{\L^\infty(T)} + \| \nabla u \|_{\Z^{\infty,q}_{-\frac{1}{2}}(T)} \leq C(T) \bigl( \| F \|_{\Z^{\infty,q}_{-\frac{1}{2}}(T)} + \| u_0 \|_\infty \bigr).
		\end{align}
		Also, for a function $\theta \colon (0,\infty) \to (0, \infty)$ with $\lim_{t\to 0} \theta(t) = 0$
		there holds
		\begin{align}
			\label{eq:wp_fp_est2}
			\| u - u_0 \|_{\L^\infty(t)} + \| \nabla u \|_{\Z^{\infty,q}_{-\frac{1}{2}}(t)} \leq C(T) \| F \|_{\Z^{\infty,q}_{-\frac{1}{2}}(t)} + \theta(t), \quad t \in (0,T].
		\end{align}
	\end{proposition}

	\begin{proof}
		The proposition follows from Proposition~\ref{prop:apriori}\ref{it:uniqueness} for uniqueness, together with Lemmas~\ref{lem:existence_LP},~\ref{lem:X} and~\ref{lem:gradient_free_evolution_BUC} below for existence along with the estimates.
	\end{proof}

	\begin{remark}
		We are not claiming that $u\in \BUC([0,T]\times \R^n)$, nor $\nabla u \in \L^2(0,T; \L^2_\loc(\R^n))$.
		Observe that convergence of $u$ to $u_0$ becomes uniform whenever $\| F \|_{\Z^{\infty,q}_{-\nicefrac{1}{2}}(t)} \to 0$ as $t\to 0$.
	\end{remark}

	\subsection{Inhomogeneous Cauchy problem}
	\label{subsec:auto_inhomogeneous}

	As part of~\eqref{eq:CP}, we consider the \emph{inhomogeneous} Cauchy problem
	\begin{align}
		\label{eq:LP}
		\tag{ICP}
		\partial_t u - \Div(A_0\nabla u) = \Div(F), \quad u(0) = 0.
	\end{align}

	\begin{lemma}[Well-posedness of bounded weak solutions to~\eqref{eq:LP}]
		\label{lem:existence_LP}
		Fix $q \in (n+2, \infty)$ and let $T \in (0,\infty)$.
		Then, for $F \in \Z^{\infty,q}_{-\nicefrac{1}{2}}(T)$ there exists a unique, bounded weak solution $u \colon (0,T) \times \R^n \to \R$ to~\eqref{eq:LP} with $u(t) \to 0$ in $\mathcal{D}'(\R^n)$.
		Moreover, there exists a constant $C(T)$, non-decreasing in $T$ and depending on the modulus of continuity and coefficient bounds of $A_0$ and $q$, such that
		\begin{align}
			\label{eq:existence_LP_bound}
			\| u \|_{\L^\infty(T)} + \| \nabla  u \|_{\Z^{\infty,q}_{-\frac{1}{2}}(T)} \leq C(T) \| F \|_{\Z^{\infty,q}_{-\frac{1}{2}}(T)}.
		\end{align}
		We set $\Sol_{A_0}(F) \coloneqq u$.
	\end{lemma}

	\begin{proof}
		Uniqueness is granted by Proposition~\ref{prop:apriori}\ref{it:uniqueness}.
		To obtain existence, we apply the results of~\cite{AB}.
		More precisely, we set $L = -\Div(A_0\nabla)$ and use the operators $\Sol^L_{\nicefrac{1}{2}}$ and $\Sol^L_0$ defined in~\cite[Sec.~4]{AB}. Then, $u \coloneqq \Sol^L_{\nicefrac{1}{2}}(F)$ solves~\eqref{eq:LP} in the weak sense by~\cite[Prop.~4.5]{AB} and satisfies $u(t) \to 0$ in $\mathcal{D}'(\R^n)$ by~\cite[Prop.~4.7]{AB}. Moreover, since we have $\nabla u = \Sol^L_0(F)$, the bound~\eqref{eq:existence_LP_bound} is a consequence of boundedness of the operators $\Sol^L_{\nicefrac{1}{2}}$ and $\Sol^L_0$ established in~\cite[Prop.~4.3]{AB}.
		In particular, this shows that $u$ is a bounded weak solution to~\eqref{eq:LP}. To apply the above results, we need that $A_0$ is a \emph{regular} coefficient function in the sense of~\cite[Def.~2.16]{AB}, but this is the case for all bounded, uniformly elliptic and uniformly continuous coefficient function as was discussed in~\cite[Ex.~2.18]{AB}.
	\end{proof}

	\subsection{Homogeneous Cauchy problem}
	\label{subsec:auto_homogeneous}

	As part of~\eqref{eq:CP}, we consider for $u_0 \in \BUC(\R^n)$ the \emph{homogeneous} Cauchy problem
	\begin{align}
		\label{eq:caloric}
		\tag{HCP}
		\partial_t u -\Div(A_0\nabla u) = 0, \quad u(0) = u_0.
	\end{align}
	The main result of this subsection is Lemma~\ref{lem:gradient_free_evolution_BUC}, which complements Proposition~\ref{prop:apriori}\ref{it:buc_grad} by a smallness condition for $\nabla u$ on small timescales.

	\subsubsection*{Heat equation}

	We start out with preliminary results for the case $A_0 = \Id$, in which~\eqref{eq:caloric} reduces to the homogeneous heat equation
	\begin{align}
		\label{eq:heat}
		\tag{HE}
		\partial_t u -\Delta u = 0, \quad u(0) = u_0.
	\end{align}
	For any $u_0 \in \BUC(\R^n)$, the heat extension $u(t,x) = \e^{t\Delta} u_0(x)$ provides us with a bounded weak solution to~\eqref{eq:heat} on $(0,\infty) \times \R^n$. Indeed, the heat semigroup is strongly continuous on $\BUC(\R^n)$, see for instance~\cite[Thm.~5.1.2~(i)]{Lunardi_Semigroup}, and $u$ satisfies~\eqref{eq:heat} weakly by virtue of an integration by parts, as it is smooth in $(0,\infty) \times \R^n$.
	Put $\Ext_\Id(u_0)(t) \coloneqq \e^{t\Delta} u_0$.

	We recall standard gradient estimates for the heat extension in terms of weighted $\L^\infty$-spaces.

	\begin{lemma}[Gradient of the heat extension]
		\label{lem:gradient_heat_extension_BUC}
		For $u_0 \in \BUC(\R^n)$, we have the estimate
		\begin{align}
			\label{eq:he_grad_bound}
			\| \nabla \Ext_\Id(u_0) \|_{\L^\infty_{-\frac{1}{2}}} \leq \frac{1}{\sqrt{2}} \| u_0 \|_\infty,
		\end{align}
		and there holds
		\begin{align}
			\label{eq:he_grad_limit}
			\lim_{t\to 0} \| \nabla \Ext_\Id(u_0) \|_{\L^\infty_{-\frac{1}{2}}(t)} = 0.
		\end{align}
	\end{lemma}

	\begin{proof}
		In order to obtain the constant $\nicefrac{1}{\sqrt{2}}$ in~\eqref{eq:he_grad_bound}, we equip $\R^n$ with the Euclidean norm. To obtain~\eqref{eq:he_grad_limit}, one can use that the convolution kernel of $\nabla \Ext_\Id$ has mean-value zero or employ an approximation argument.
	\end{proof}

	\subsubsection*{General initial value problems}

	We return to the homogeneous Cauchy problem~\eqref{eq:caloric}. Bounded weak solutions can be constructed using the heat semigroup generated by $L \coloneqq -\Div(A_0\nabla)$, and they possess the regularity established in Proposition~\ref{prop:apriori}.

	\begin{lemma}[Free evolution]
		\label{lem:X}
		For all $u_0 \in \BUC(\R^n)$ there exists a unique, bounded weak solution to~\eqref{eq:caloric}, denoted by $\Ext_{A_0}(u_0)$, which is uniformly continuous on $[0,\infty) \times \R^n$ and satisfies the estimate
		\begin{align}
			\| \Ext_{A_0}(u_0) \|_{\L^\infty} \leq \| u_0 \|_\infty.
		\end{align}
	\end{lemma}

	We derive a representation formula for $\Ext_{A_0}(u_0)$ in terms of $\Ext_\Id(u_0)$, showing the generic initial value problem~\eqref{eq:caloric} as a \enquote{perturbation} of the heat equation.

	\begin{lemma}[Representation]
		\label{lem:caloric_repr}
		For $u_0 \in \BUC(\R^n)$, there holds the representation formula
		\begin{align}
			\label{eq:caloric_extension_repr}
			\Ext_{A_0}(u_0) = \Ext_\Id(u_0) + \Sol_{A_0}((A_0-\Id)\nabla \Ext_\Id(u_0)),
		\end{align}
		where $\Sol_{A_0}((A_0-\Id)\nabla \Ext_\Id(u_0))$ is defined in Lemma~\ref{lem:existence_LP}.
	\end{lemma}

	\begin{proof}
		Lemma~\ref{lem:existence_LP} applies with $F \coloneqq (A_0-\Id)\nabla \Ext_\Id(u_0)$ thanks to Lemma~\ref{lem:gradient_heat_extension_BUC} and the nesting property of weighted $\Z$-spaces. Hence, the right-hand side of~\eqref{eq:caloric_extension_repr} is a bounded weak solution to~\eqref{eq:caloric}.
		So is the left-hand side of~\eqref{eq:caloric_extension_repr}, therefore the claim follows by the uniqueness result in Proposition~\ref{prop:apriori}\ref{it:uniqueness}.
	\end{proof}

	\begin{lemma}[Gradient of the free evolution]
		\label{lem:gradient_free_evolution_BUC}
		Fix $q\in (1, \infty)$ and let $T \in (0,\infty)$.
		Then, there exists a finite constant $C(T)$, non-decreasing in $T$ and depending on the modulus of continuity and coefficient bounds of $A_0$ and $q$, such that for all $u_0 \in \BUC(\R^n)$, we have
		\begin{align}
			\label{eq:ivp_grad_bound}
			\| \nabla \Ext_{A_0}(u_0) \|_{\Z^{\infty,q}_{-\frac{1}{2}}(T)} \leq C(T) \| u_0 \|_\infty.
		\end{align}
		Moreover, there holds
		\begin{align}
			\label{eq:ivp_grad_limit}
			\lim_{t\to 0} \| \nabla \Ext_{A_0}(u_0) \|_{\Z^{\infty,q}_{-\frac{1}{2}}(t)} = 0.
		\end{align}
	\end{lemma}

	\begin{proof}
		By the nesting property of weighted $\Z$-spaces, it suffices to consider $q\in (n+2,\infty)$. By Lemma~\ref{lem:caloric_repr}, we have the representation
		\begin{align}
			\nabla \Ext_{A_0}(u_0) = \nabla \Ext_\Id(u_0) + \nabla \Sol_{A_0}((A_0-\Id)\nabla \Ext_\Id(u_0)).
		\end{align}
		We start with~\eqref{eq:ivp_grad_limit}.
		For $\nabla \Ext_\Id(u_0)$, the claim follows from Lemma~\ref{lem:gradient_heat_extension_BUC} and the nesting property of weighted $\Z$-spaces. As for $\nabla \Sol_{A_0}((A_0-\Id)\nabla \Ext_\Id(u_0))$, we use Lemma~\ref{lem:existence_LP} with a constant $C = C(1)$, boundedness of $A_0$ and Lemma~\ref{lem:gradient_heat_extension_BUC}, to find
		\begin{align}
			\limsup_{t\to 0} \| \nabla \Sol_{A_0}((A_0-\Id)\nabla \Ext_\Id(u_0)) \|_{\Z^{\infty,q}_{-\frac{1}{2}}(t)} \lesssim \limsup_{t\to 0} \| \nabla \Ext_\Id(u_0) \|_{\Z^{\infty,q}_{-\frac{1}{2}}(t)} = 0.
		\end{align}
		Combining both facts gives~\eqref{eq:ivp_grad_limit}. By a similar reasoning, we deduce~\eqref{eq:ivp_grad_bound}. The dependence on $T$ of the constant $C(T)$ stems from Lemma~\ref{lem:existence_LP}.
	\end{proof}

	\section{The quasilinear problem}
	\label{sec:ql}

	Throughout this section, let $a \colon [0,\infty) \times \R^n \times O \to \R^{n \times n}$ with $O \subseteq \R$ open satisfy Assumption~\ref{ass:a}.
	We are going to prove Theorem~\ref{thm:main} and Corollary~\ref{cor:Loo} in this section.

	\subsection{A priori properties}
	\label{subsec:ql_apriori}

	We show a priori properties of bounded weak solutions to~\eqref{eq:ql}.

	\begin{proposition}[A priori regularity of bounded weak solutions to~\eqref{eq:ql}]
		\label{prop:ql_apriori}
		Let $T \in (0,\infty]$ and $u_0 \in \BUC(\R^n)$ with $\overline{\Ran(u_0)} \subseteq O$. For a bounded weak solution $u \colon (0,T) \times \R^n \to O$  to~\eqref{eq:ql}, the following hold:
		\begin{enumerate}[(i)]
			\item\label{it:ql1} $u$ is uniformly continuous on $[0,T) \times \R^n$. In particular, if $T$ is finite, then $u(T)$ exists and belongs to $\BUC(\R^n)$.
			\item\label{it:ql2} $\overline{\Ran(u)} = \overline{\Ran(u_0)}$, and hence $\| u \|_{\L^\infty(T)} = \| u_0 \|_\infty$.
			\item\label{it:ql2.5} There exists a finite constant $C$ depending on $\overline{\Ran(u_0)}$, such that we have the Carleson measure estimate
				\begin{align}
					\sup_{x\in \R^n} \sup_{t \in (0,T]} \bigg( \int_0^t \fint_{\B(x,\sqrt{t})} |\nabla u(s,y)|^2 \d y \d s \bigg)^\frac{1}{2} \leq C \| u_0 \|_\infty.
			\end{align}
			\item\label{it:ql3} For all $T' \in (0,T] \cap (0,\infty)$ and $q \in (1,\infty)$, we have $\nabla u \in \Z^{\infty,q}_{-\nicefrac{1}{2}}(T')$ with
			\begin{align}
				\lim_{t\to 0} \| \nabla u \|_{\Z^{\infty,q}_{-\frac{1}{2}}(t)} = 0.
			\end{align}
		\end{enumerate}
	\end{proposition}

	\begin{remark}
		\label{rem:ql_apriori}
		Proposition~\ref{prop:ql_apriori}\ref{it:ql2} remains true for $u_0 \in \L^\infty(\R^n)$ with $\overline{\vphantom{t}\Ran}(u_0) \subseteq O$ up to replacing the closure of the range by the convex hull of the essential range. In particular, $u$ and $u_0$ have the same essential infimum and essential supremum. The Carleson measure estimate in~\ref{it:ql2.5} also remains true for $u_0 \in \L^\infty(\R^n)$.
	\end{remark}

	As a preparation for its proof, we investigate coefficient functions related to~\eqref{eq:ql}.

	\begin{lemma}
		\label{lem:fp_uniformly_elliptic}
		The following hold:
		\begin{enumerate}[(i)]
			\item\label{it:fp_uniformly_elliptic_1} For all $v_0 \in \BUC(\R^n)$ with $\overline{\Ran(v_0)} \subseteq K$ for some compact $K \subseteq O$, the coefficient function $A_0(x) \coloneqq a(0,x,v_0(x))$ is bounded, uniformly continuous and uniformly elliptic on $\R^n$.
			The modulus of continuity and coefficient bounds of $A_0$ depend on $K$.
			\item\label{it:fp_uniformly_elliptic_2} For all $T \in (0,\infty]$ and $v \colon (0,T) \times \R^n \to O$ measurable, bounded with $\overline{\vphantom{t}\Ran}(v) \subseteq K$ for some compact $K \subseteq O$, $A(t,x) \coloneqq a(t,x,v(t,x))$ is measurable, bounded and uniformly elliptic on $(0,T) \times \R^n$.
			The coefficient bounds of $A$ depend on $K$.
			\item\label{it:fp_uniformly_elliptic_3} For all $T \in (0,\infty)$ and $v \colon (0,T) \times \R^n \to O$ bounded, uniformly continuous with $\overline{\Ran(v)} \subseteq K$ for some compact $K \subseteq O$, $A(t,x) \coloneqq a(t,x,v(t,x))$ is uniformly continuous on $(0,T) \times \R^n$.
			The modulus of continuity of $A$ depends on $K$.
		\end{enumerate}
	\end{lemma}

	\begin{proof}
		We start with~\ref{it:fp_uniformly_elliptic_1}.
		First, since $\overline{\Ran(v_0)} \subseteq K$ for some compact $K \subseteq O$, boundedness derives from Assumptions~\ref{ass:a}\ref{it:a2} and~\ref{ass:a}\ref{it:a3} as
		\begin{align}
			|A_0(x)| &\leq |a(0,x,v_0(x))-a(0,x,0)| + |a(0,x,0)| \\
			&\leq C_L(0,K) \| v_0 \|_\infty + C_E(0).
		\end{align}
		Second, uniform continuity follows from Assumption~\ref{ass:a}\ref{it:a4} by composition.
		Third, uniform ellipticity is granted by Assumption~\ref{ass:a}\ref{it:a1}.

		As for~\ref{it:fp_uniformly_elliptic_2}, we can follow the lines of~\ref{it:fp_uniformly_elliptic_1}
		as the Assumptions~\ref{ass:a}\ref{it:a1} and~\ref{ass:a}\ref{it:a2} are still applicable using the assumption $\overline{\vphantom{t}\Ran}(v) \subseteq K$ for some compact $K \subseteq O$.
		Invoking also Assumption~\ref{ass:a}\ref{it:a3}, the bound for $A$ reads
		\begin{align}
			\label{eq:ql_apriori_bounded}
			|A(t,x)| \leq C_L(T,K) \| v \|_{\L^\infty(T)} + C_E(T).
		\end{align}

		Finally,~\ref{it:fp_uniformly_elliptic_3} follows from Assumption~\ref{ass:a}\ref{it:a4} by composition.
		Note that Assumption~\ref{ass:a}\ref{it:a4} imports the restriction $T \in (0,\infty)$.
	\end{proof}

	\begin{proof}[Proof of Proposition~\ref{prop:ql_apriori}]
		For~\ref{it:ql1}, note that $K \coloneqq \overline{\Ran(u)}$ is compact and contained in $O$.
		Define the non-autonomous coefficient function $A(t,x) \coloneqq a(t,x,u(t,x))$ on $(0,T) \times \R^n$.
		By Lemma~\ref{lem:fp_uniformly_elliptic}\ref{it:fp_uniformly_elliptic_2}, $A$ is measurable, bounded and uniformly elliptic.
		Hence, Proposition~\ref{prop:apriori}\ref{it:buc} yields uniform continuity of $u$ on $[0,T) \times \R^n$, and $u(T) \in \BUC(\R^n)$ if $T$ is finite.

		As for~\ref{it:ql2}, taking Remark~\ref{rem:conv_hull} into account, we apply Proposition~\ref{prop:apriori}\ref{it:range} with $A$ as above.

		Next,~\ref{it:ql2.5} follows from Proposition~\ref{prop:apriori}\ref{it:Carleson} with $A$ as above in conjunction with~\ref{it:ql2}.

		It remains to show~\ref{it:ql3}. By the nesting property of weighted $\Z$-spaces, it suffices to consider $q \in (n+2,\infty)$. First, by Proposition~\ref{prop:apriori}\ref{it:buc_grad}, we have $\nabla u \in \Z^{\infty,q}_{-\nicefrac{1}{2}}(T')$. Second, set $A_0(x) \coloneqq a(0,x,u_0(x))$ on $\R^n$ and note that $A_0$ is bounded, uniformly continuous and uniformly elliptic by Lemma~\ref{lem:fp_uniformly_elliptic}\ref{it:fp_uniformly_elliptic_1}.
		Then, $u$ is the bounded weak solution to
		\begin{align}
			\partial_t u - \Div(A_0 \nabla u) = \Div((A-A_0)\nabla u), \quad u(0) = u_0.
		\end{align}
		Set $F \coloneqq (A-A_0)\nabla u$. By Assumption~\ref{ass:a}\ref{it:a2} and~\ref{it:ql2}, we have for almost every $(t,x) \in (0,T) \times \R^n$ the estimate
		\begin{align}
			|A(t,x) - A_0(x)| &\leq |a(t,x,u(t,x)) - a(t,x,u_0(x))| + |a(t,x,u_0(x)) - a(0,x,u_0(x))| \\
			&\leq C_L(T,K) \| u - u_0 \|_{\L^\infty(t)} + |a(t,x,u_0(x)) - a(0,x,u_0(x))|.
		\end{align}
		Hence, using $u(0) = u_0$ in conjunction with~\ref{it:ql1} as well as Assumption~\ref{ass:a}\ref{it:a4}, we have $\lim_{S\to 0} \| A - A_0 \|_{\L^\infty(S)} = 0$.
		Thus, since $\nabla u \in \Z^{\infty,q}_{-\nicefrac{1}{2}}(T')$, we get $\lim_{S\to 0} \| F \|_{\Z^{\infty,q}_{-\nicefrac{1}{2}}(S)} = 0$. Consequently, Proposition~\ref{prop:cp_wp} yields $\lim_{t\to 0} \| \nabla u \|_{\Z^{\infty,q}_{-\nicefrac{1}{2}}(t)} = 0$ as desired.
	\end{proof}

	\subsection{Local existence}
	\label{subsec:ql_existence}

	Fix $u_0 \in \BUC(\R^n)$ with $\overline{\Ran(u_0)} \subseteq O$.
	We establish existence of a local, bounded weak solution to~\eqref{eq:ql} using a fixed-point argument.

	\begin{proposition}[Local existence to~\eqref{eq:ql}]
		\label{prop:ql_existence}
		For some $T \in (0,1)$, there exists a bounded weak solution $u \colon (0,T) \times \R^n \to O$ to~\eqref{eq:ql}.
	\end{proposition}

	Fix $q \in (n+2,\infty)$. For $r,T \in (0,1)$, define the complete metric set
	\begin{align}
		\X_{u_0}(r, T) \coloneqq \bigl\{ v \in \L^\infty(T) \,; \| v - u_0 \|_{\L^\infty(T)} \leq r, \| \nabla v \|_{\Z^{\infty,q}_{-\frac{1}{2}}(T)} \leq r \bigr\}
	\end{align}
	for the distance $\dist(u,v) \coloneqq \| u-v \|_{\X(T)}$, where
	\begin{align}
		\| v \|_{\X(T)} \coloneqq \| v \|_{\L^\infty(T)} + \| \nabla v \|_{\Z^{\infty,q}_{-\frac{1}{2}}(T)}.
	\end{align}
	The following lemma ensures $\overline{\vphantom{t}\Ran}(v) \subseteq O$ when $v \in\X_{u_0}(r, T)$.

	\begin{lemma}
		\label{lem:fp_O_valued}
		There exists $r^\sharp \in (0,1)$ such that, for all $r \in (0, r^\sharp)$, $T \in (0,1)$ and $v \in \X_{u_0}(r, T)$, we have $\overline{\vphantom{t}\Ran}(v) \subseteq K$ for some compact $K\subseteq O$ depending on $r^\sharp$.
	\end{lemma}

	\begin{proof}
		As $\overline{\Ran(u_0)}$ is compact and contained in the open set $O$, $r^\sharp \coloneqq \tfrac{1}{2}\dist(\Ran(u_0), \partial O) > 0$.
		By definition of $\X_{u_0}(r, T)$, we have $|v(t,x) - u_0(x)| \leq \| v - u_0 \|_{\L^\infty(T)} \leq r < r^\sharp$ for almost every $(t,x) \in (0,T) \times \R^n$, which yields the claim.
	\end{proof}

	We setup a fixed-point argument.
	For every $v \in \X_{u_0}(r, T)$ consider on $(0,T) \times \R^n$ the auxiliary autonomous linear problem
	\begin{align}
		\label{eq:fp}
		\tag{FP}
		\partial_t u - \Div(a(0,x,u_0) \nabla u) = \Div((a(t,x,v) - a(0,x,u_0)) \nabla v), \quad u(0) = u_0.
	\end{align}
	Set $A_0(x) \coloneqq a(0,x,u_0(x))$ on $\R^n$, and $A_v(t,x) \coloneqq a(t,x,v(t,x))$ on $(0,T) \times \R^n$. By Lemma~\ref{lem:fp_uniformly_elliptic}\ref{it:fp_uniformly_elliptic_1} and~\ref{it:fp_uniformly_elliptic_2}, $A_v$ is bounded and $A_0$ is bounded, uniformly continuous and uniformly elliptic. Hence, Proposition~\ref{prop:cp_wp} with $F \coloneqq (A_v-A_0)\nabla v$ yields a unique, bounded weak solution $u\colon (0,T) \times \R^n \to \R$ to~\eqref{eq:fp}.
	Put $\Theta(v) \coloneqq u$.

	\begin{lemma}[Self-mapping and strict contraction]
		\label{lem:fp_self-mapping}
		There is $r^* \in (0,r^\sharp)$ such that, for all $r \in (0,r^*)$, one can find $T^* = T^*(r) \in (0,1)$ depending on $r$, such that $\Theta$ maps $\X_{u_0}(r,T)$ into itself and is a strict contraction for all $T \in (0, T^*)$.
	\end{lemma}

	\begin{proof}
		Define $A_0(x) \coloneqq a(0,x,u_0(x))$ on $\R^n$, and $A_u(t,x) \coloneqq a(t,x,u(t,x))$ and $A_v(t,x) \coloneqq a(t,x,v(t,x))$ on $(0,T) \times \R^n$.
		We consider the self-mapping and strict contraction properties in two separate steps.
		Let $K \subseteq O$ be the compact set of Lemma~\ref{lem:fp_O_valued}.

		\textbf{Step 1}: self-mapping property. Let $v \in \X_{u_0}(r,T)$.
		By Proposition~\ref{prop:cp_wp}, there exists a constant $C = C(1)$, depending on the modulus of continuity and coefficient bounds of $A_0$ and $q$, such that
		\begin{align}
			\label{eq:self-mapping1}
			\begin{aligned}
				\| \Theta(v) - u_0 \|_{\L^\infty(T)} + \| \nabla \Theta(v) \|_{\Z^{\infty,q}_{-\frac{1}{2}}(T)} \leq C \| (A_v-A_0) \nabla v \|_{\Z^{\infty,q}_{-\frac{1}{2}}(T)} + \theta(T).
			\end{aligned}
		\end{align}
		Then, for almost every $(t,x) \in (0,T) \times \R^n$, use Assumption~\ref{ass:a}\ref{it:a2} with $r \leq r^\sharp$ to estimate
		\begin{align}
			\label{eq:self-mapping1.5}
			\begin{aligned}
				&|A_v(t,x) - A_0(x)| \\
				&\qquad \leq |a(t,x,v(t,x)) - a(t,x,u_0(x))| + |a(t,x,u_0(x)) - a(0,x,u_0(x))| \\
				&\qquad \leq C_L(1,K) r + \| a(t,x,u_0(x)) - a(0,x,u_0(x)) \|_{\L^\infty(T)}.
			\end{aligned}
		\end{align}
		Now, first take $r^* \leq r^\sharp$ small enough such that
		\begin{align}
			\label{eq:self-mapping1.6}
			C C_L(1,K) r^* \leq \tfrac{1}{6}.
		\end{align}
		Second, using Assumption~\ref{ass:a}\ref{it:a4}, take $T^*$ small enough such that
		\begin{align}
			\label{eq:self-mapping1.75}
			C \| a(t,x,u_0(x)) - a(0,x,u_0(x)) \|_{\L^\infty(T^*)} \leq \tfrac{1}{6}.
		\end{align}
		Hence, when $r \leq r^*$ and $T \leq T^*$,~\eqref{eq:self-mapping1.5} yields
		\begin{align}
			\label{eq:self-mapping2}
			C \| A_v-A_0 \|_{\L^\infty(T)} \leq \tfrac{1}{3}.
		\end{align}
		Consequently,
		\begin{align}
			\label{eq:self-mapping3}
			C \| (A_v - A_0) \nabla v \|_{\Z^{\infty,q}_{-\frac{1}{2}}(T)} \leq \tfrac{1}{3} \| \nabla v \|_{\Z^{\infty,q}_{-\frac{1}{2}}(T)}
			\leq \tfrac{r}{3}.
		\end{align}
		Eventually, since $\theta(S) \to 0$ as $S \to 0$, we can take $T^*=T^*(r)$ small enough such that $\theta(T) \leq \tfrac{2r}{3}$ for all $T\in (0,T^*)$. Plugging this and~\eqref{eq:self-mapping3} back into~\eqref{eq:self-mapping1} yields $\Theta(v) \in \X_{u_0}(r,T)$. This completes this step.

		\textbf{Step 2}: strict contraction property.
		By Assumption~\ref{ass:a}\ref{it:a2}, we find
		\begin{align}
			\label{eq:contraction0.25}
			\| A_u - A_v \|_{\L^\infty(T)} \leq C_L(1,K) \| u - v \|_{\L^\infty(T)}.
		\end{align}
		Observe that $w \coloneqq\Theta(u) - \Theta(v)$ is the unique, bounded weak solution on $(0,T) \times \R^n$ to
		\begin{align}
				\partial_t w - \Div(A_0 \nabla w) = \Div\Big[(A_u - A_v) \nabla u + (A_v - A_0)\nabla (u-v)\Big], \quad
				w(0) = 0.
		\end{align}
		Therefore, using Proposition~\ref{prop:cp_wp} with $C = C(1)$ as in Step~1,~\eqref{eq:contraction0.25},~\eqref{eq:self-mapping2},~\eqref{eq:self-mapping1.6} and $\| \nabla u \|_{\Z^{\infty,q}_{-\nicefrac{1}{2}}(T)} \leq r^*$, we have
		\begin{align}
			\label{eq:contraction1}
			&\| \Theta(u) - \Theta(v) \|_{\X(T)} \\
			&\qquad\leq{} C \| (A_u - A_v) \nabla u \|_{\Z^{\infty,q}_{-\frac{1}{2}}(T)} + C\|  (A_v - A_0)\nabla (u-v) \|_{\Z^{\infty,q}_{-\frac{1}{2}}(T)} \\
			&\qquad\leq{} C C_L(1,K) \| u - v \|_{\L^\infty(T)} \| \nabla u \|_{\Z^{\infty,q}_{-\frac{1}{2}}(T)} + \tfrac{1}{3} \| \nabla (u-v) \|_{\Z^{\infty,q}_{-\frac{1}{2}}(T)} \\
			&\qquad\leq{} \tfrac{1}{2} \| u-v \|_{\X(T)}.
		\end{align}
		Hence, $\Theta$ is a strict contraction.
	\end{proof}

	Now, we can proceed with the proof of Proposition~\ref{prop:ql_existence}.

	\begin{proof}[Proof of Proposition~\ref{prop:ql_existence}]
		Let $r \in (0,r^*)$ and $T\in (0,1)$ meeting the requirements of Lemma~\ref{lem:fp_self-mapping}. The map $\Theta$ admits a fixed point $u \in \X_{u_0}(r,T)$ by Banach's contraction theorem.
		Hence, by definition, $u \colon (0,T) \times \R^n \to \R$ is the unique, bounded weak solution to
		\begin{align}
			\partial_t u - \Div(a(0,x,u_0) \nabla u) = \Div((a(t,x,u) - a(0,x,u_0))\nabla u), \quad u(0) = u_0.
		\end{align}
		Rearranging the terms reveals that $u$ is a bounded weak solution to
		\begin{align}
			\partial_t u - \Div(a(t,x,u) \nabla u) = 0, \quad u(0) = u_0.
		\end{align}
		By Lemma~\ref{lem:fp_O_valued}, $\overline{\vphantom{t}\Ran}(u) \subseteq O$.
		Hence, $u$ is a bounded weak solution to~\eqref{eq:ql}.
	\end{proof}

	\subsection{Uniqueness}
	\label{subsec:ql_uniqueness}

	Fix $u_0 \in \BUC(\R^n)$ with $\overline{\Ran(u_0)} \subseteq O$.
	We investigate uniqueness of  bounded weak solutions to~\eqref{eq:ql}.

	\begin{proposition}[Uniqueness of~\eqref{eq:ql}]
		\label{prop:ql_uniqueness}
		Let $T\in (0, \infty]$. Then, there exists at most one bounded weak solution $u \colon (0,T) \times \R^n \to O$ to~\eqref{eq:ql}.
	\end{proposition}

	\begin{proof}
		Let $u,v \colon (0,T) \times \R^n \to O$ be two bounded weak solutions to~\eqref{eq:ql}.
		By Proposition~\ref{prop:ql_apriori}\ref{it:ql1}, $u$ and $v$ are uniformly continuous on $[0,T) \times \R^n$ with $u(0) = v(0) = u_0$. We want to show that $u=v$ on $[0,T) \times \R^n$. Define the branching time
		\begin{align}
			\tau \coloneqq \sup \Bigl\{ t \in [0,T) \,; u = v \text{ on } [0,t] \times \R^n \Bigr\}.
		\end{align}
		Since $u(0) = v(0)$, the set on the right-hand side is non-empty. Suppose for the sake of contradiction that $\tau < T$.
		Define
		\begin{alignat*}{2}
			\tilde a(t,x,y) &\coloneqq a(t+\tau,x,y), & \qquad (t,x,y) &\in [0,\infty) \times \R^n \times O, \\
			\tilde u(t,x) &\coloneqq u(t+\tau,x), & \qquad (t,x) &\in [0,T-\tau) \times \R^n, \\
			\tilde v(t,x) &\coloneqq v(t+\tau,x), & \qquad (t,x) &\in [0,T-\tau) \times \R^n, \\
			w_0(x) &\coloneqq u(\tau,x), & \qquad x &\in \R^n.
		\end{alignat*}
		By definition of $\tau$ and uniform continuity of $u$ and $v$, we have $u(\tau) = v(\tau)$, hence $\tilde u(0) = \tilde v(0) = w_0$.
		It follows that $\tilde u$ and $\tilde v$ are bounded weak solutions on $(0,T-\tau) \times \R^n$ to
		\begin{align}
			\label{eq:uniqueness_ql}
			\partial_t w - \Div(\tilde a(t,x,w) \nabla w) = 0, \quad w(0) = w_0.
		\end{align}
		Clearly, $\tilde a$ satisfies Assumption~\ref{ass:a} and we have $\overline{\Ran(w_0)} \subseteq O$.
		Hence, the local well-posedness theory from Section~\ref{subsec:ql_existence} applies to~\eqref{eq:uniqueness_ql}.
		We show $\tilde u,\tilde v \in \X_{w_0}(r, \delta)$ for any $r \in (0,r^*)$ and some $\delta = \delta(r) \in (0,\min(T^*, T-\tau))$ as in Lemma~\ref{lem:fp_self-mapping}. Then, since $\tilde u$ and $\tilde v$ solve~\eqref{eq:uniqueness_ql}, they are fixed points of $\Theta$ and thus coincide on $[0,\delta] \times \R^n$. This means $u = v$ on $[0, \tau + \delta] \times \R^n$, contradicting the definition of $\tau$.
		Now, since $\tilde u(0) = w_0$ and $\tilde u$ is uniformly continuous on $[0,T-\tau) \times \R^n$ by Proposition~\ref{prop:ql_apriori}\ref{it:ql1}, we find $\| \tilde u - w_0 \|_{\L^\infty(\delta)} \leq r$ for $\delta$ small enough depending on $\tilde u$ and $r$. Similarly, $\| \nabla \tilde u \|_{\Z^{\infty,q}_{-\nicefrac{1}{2}}(\delta)} \leq r$ follows from Proposition~\ref{prop:ql_apriori}\ref{it:ql3} by taking $\delta$ sufficiently small. Hence, $\tilde u \in \X_{w_0}(r, \delta)$. By the same argument, potentially decreasing $\delta$ further depending on $\tilde v$, we also find $\tilde v \in \X_{w_0}(r, \delta)$, which completes the proof.
	\end{proof}

	\begin{remark}
		Our uniqueness proof relies on weighted $\Z$-spaces in a crucial way. Indeed, it uses a smallness condition on the gradient of solutions to~\eqref{eq:ql} granted by Proposition~\ref{prop:ql_apriori}\ref{it:ql3}.
	\end{remark}

	\subsection{Global solutions and conclusion of Theorem~\ref{thm:main}}
	\label{subsec:ql_global}

	We combine the existence and uniqueness results from Sections~\ref{subsec:ql_existence} and~\ref{subsec:ql_uniqueness} with the usual maximization procedure to build a maximal solution to~\eqref{eq:ql}. Using a blow-up argument, this solution turns out to be global, allowing us to conclude Theorem~\ref{thm:main}.

	\begin{proof}[Proof of Theorem~\ref{thm:main}]
		We subdivide the proof into several steps.

		\textbf{Step 1}: construction of a maximal, bounded weak solution.
		Consider the set
		\begin{align}
			M \coloneqq \bigg\{ (v,T) \,; T \in (0,\infty], \; v \colon (0,T) \times \R^n \to O \text{ is a bounded weak solution to \eqref{eq:ql}} \bigg\}.
		\end{align}
		By Proposition~\ref{prop:ql_existence}, $M$ is non-empty. Define the maximal existence time
		\begin{align}
			\tau \coloneqq \sup \{ T \,; (v,T) \in M \},
		\end{align}
		and construct the maximal solution $u \colon (0,\tau) \times \R^n \to O$ by $u(t) \coloneqq v(t)$, where $t\in (0,\tau)$ and $(v,T)$ is \emph{any} solution in $M$ satisfying $T > t$. By the uniqueness result in Proposition~\ref{prop:ql_uniqueness}, this definition is unambiguous.
		For all $(v,T) \in M$ we have $\| v \|_{\L^\infty(T)} \leq \| u_0 \|_\infty$ and $\overline{\Ran(v)} \subseteq \overline{\Ran(u_0)}$ by Proposition~\ref{prop:ql_apriori}\ref{it:ql2}, hence $\| u \|_{\L^\infty(\tau)} \leq \| u_0 \|_\infty$ and $\overline{\Ran(u)} \subseteq O$.
		Therefore, $u$ is by construction a bounded weak solution to~\eqref{eq:ql} on $(0,\tau) \times \R^n$.

		\textbf{Step 2}: existence and uniqueness of a global, bounded weak solution.
		To see that $u$ is a global solution, suppose for the sake of contradiction that $\tau$ were finite. Then, $u(\tau) \in \BUC(\R^n)$ is well-defined by Proposition~\ref{prop:ql_apriori}\ref{it:ql1}.
		Moreover, there holds $\overline{\Ran(u(\tau))} \subseteq O$ by Proposition~\ref{prop:ql_apriori}\ref{it:ql2} and continuity.
		Consider
		\begin{align}
			\label{eq:restart_ql}
			\partial_t w - \Div(a(t,x,w) \nabla w) = 0, \quad w(\tau) = u(\tau),
		\end{align}
		on $(\tau, \infty) \times \R^n$. Our local existence result in Proposition~\ref{prop:ql_existence} applies to this problem when we consider $\tau$ as the new temporal origin. Consequently, there exists a bounded weak solution $w \colon (\tau, \sigma) \times \R^n \to O$ to~\eqref{eq:restart_ql} with $\sigma > \tau$ which is uniformly continuous on $[\tau, \sigma) \times \R^n$.
		Using Proposition~\ref{prop:ql_apriori}\ref{it:ql3} and $u(\tau) = w(\tau)$, we may apply Lemma~\ref{lem:Lions_cont} to glue $u$ and $w$ together in order to obtain a bounded weak solution $\tilde u \colon (0,\sigma) \times \R^n \to O$ to~\eqref{eq:ql}. But this means $(\tilde u, \sigma) \in M$, contradicting the definition of $\tau$. Therefore, $\tau = \infty$ and $u$ is a global, bounded weak solution to~\eqref{eq:ql}. By Proposition~\ref{prop:ql_uniqueness}, it is moreover unique.

		\textbf{Step 3}: properties of the global solution. The regularity properties of $u$ were established in Proposition~\ref{prop:ql_apriori}\ref{it:ql1},~\ref{it:ql2.5} and~\ref{it:ql3}. Moreover, the limit $t \to 0$ as well as the range condition follow from Proposition~\ref{prop:ql_apriori}\ref{it:ql1} and~\ref{it:ql2}. It remains to consider the limit $t\to \infty$ when $\lim_{|x| \to \infty} u_0(x) = c$ exists. Let $\eps > 0$, then there exists a compact $K \subseteq \R^n$ such that $\sup_{x \not\in K} |u_0(x) - c| \leq \eps$. Set $A(t,x) \coloneqq a(t,x,u(t,x))$ on $(0,\infty) \times \R^n$ and recall that $A$ is measurable, bounded and uniformly elliptic by Lemma~\ref{lem:fp_uniformly_elliptic}\ref{it:fp_uniformly_elliptic_2}. Similar to Proposition~\ref{prop:apriori}, using Aronson's fundamental solution $\Gamma_A(t,x,0,y)$ associated with the coefficient function $A$, we can represent $u$ for all $(t,x) \in (0,\infty) \times \R^n$ by
		\begin{align}
			u(t,x) = \int_{\R^n} \Gamma_A(t,x,0,y) u_0(y) \d y.
		\end{align}
		Using that $\int_{\R^n} \Gamma_A(t,x,0,y) \d y = 1$ and the Gaussian kernel bounds~\eqref{eq:Gaussian}, we find
		\begin{align}
			|u(t,x) - c| \leq \eps + \int_K \Gamma_A(t,x,0,y) |u_0(y) - c| \d y
			\leq \eps + C(A) t^{-\frac{n}{2}} \| u_0 - c \|_{\L^1(K)},
		\end{align}
		where the finite constant $C(A)$ depends on $A$ through its Gaussian kernel bounds.
		Since $C(A) \| u_0 - c \|_{\L^1(K)}$ is finite and independent of $t$ and $x$, we find $\| u(t) - c \|_\infty \leq 2\eps$ for $t$ sufficiently large. This means $u(t) \to c$ in $\L^\infty(\R^n)$ as $t\to \infty$, which completes the proof.
	\end{proof}

	\begin{remark}
		Although our solution $u$ is global, we do not establish $\| \nabla u \|_{\Z^{\infty,q}_{-\nicefrac{1}{2}}} < \infty$. We only know $\| \nabla u \|_{\Z^{\infty,q}_{-\nicefrac{1}{2}}(T)} < \infty$ for all $T \in (0,\infty)$.
	\end{remark}

	\subsection{Existence with bounded data}
	\label{subsec:ql_existence_Loo}

	By combining Theorem~\ref{thm:main} with a compactness argument, we obtain for all $u_0 \in\L^\infty(\R^n)$ existence of a global, bounded weak solutions to~\eqref{eq:ql}.

	\begin{proof}[Proof of Corollary~\ref{cor:Loo}]
		Let $\Phi \in \Cont_\cc^\infty(\B(0,1))$ with $\Phi \geq 0$ and $\int_{\R^n} \Phi \d x = 1$ and put $\Phi_\eps(x) \coloneqq \eps^{-n} \Phi(\nicefrac{x}{\eps})$ on $\R^n$, $\eps \in (0,1)$. Define $u_{0,\eps} \coloneqq u_0 \ast \Phi_\eps$, then $u_{0,\eps} \in \BUC(\R^n)$ has a modulus of continuity depending on $\eps$, but we have the uniform bound
		\begin{align}
			\label{eq:cor_Loo1}
			\| u_{0,\eps} \|_\infty \leq \| u_0 \|_\infty.
		\end{align}
		Similarly to Lemma~\ref{lem:fp_O_valued}, there is $\eps_0 \in (0,1)$ small enough and a compact $K \subseteq O$ such that, for all $\eps \in (0,\eps_0)$, we have
		\begin{align}
			\label{eq:cor_Loo2}
			\overline{\Ran(u_{0,\eps})} \subseteq K.
		\end{align}
		Hence, for every $\eps \in (0,\eps_0)$, Theorem~\ref{thm:main} provides us with a global, bounded weak solution $u_\eps \colon (0,\infty) \times \R^n \to O$ to
		\begin{align}
			\label{eq:cor_Loo2.5}
			\partial_t u_\eps - \Div(a(t,x,u_\eps) \nabla u_\eps) = 0, \quad u_\eps(0) = u_{0,\eps}.
		\end{align}
		By Proposition~\ref{prop:ql_apriori}\ref{it:ql2},~\eqref{eq:cor_Loo1} and~\eqref{eq:cor_Loo2}, we have
		\begin{align}
			\label{eq:cor_Loo3}
			\| u_\eps \|_{\L^\infty} \leq \| u_{0,\eps} \|_\infty \leq \| u_0 \|_\infty
		\end{align}
		as well as
		\begin{align}
			\label{eq:cor_Loo4}
			\overline{\Ran(u_{\eps})} = \overline{\Ran(u_{0,\eps})} \subseteq K \subseteq O.
		\end{align}
		Let $R,T \in (0,\infty)$ and set $B_R \coloneqq \B(0,R)$ and $Q_{T,R} \coloneqq (0,T) \times B_R$.
		By Lemma~\ref{lem:fp_uniformly_elliptic}\ref{it:fp_uniformly_elliptic_2}, the coefficient function $A_\eps(t,x) \coloneqq a(t,x,u_\eps(t,x))$ is measurable, bounded and uniformly elliptic on $(0,T) \times \R^n$ with coefficient bounds depending only on $K$, thus independent of $\eps$ by~\eqref{eq:cor_Loo4}.
		Hence, Proposition~\ref{prop:ql_apriori}\ref{it:ql1} and~\eqref{eq:cor_Loo3} yield
		\begin{align}
			\iint_{Q_{T,R}} |\nabla u_\eps|^2 \d x \d t \leq C(T,R,K) \| u_0 \|_\infty^2.
		\end{align}
		This and~\eqref{eq:cor_Loo3} imply that $\{ u_\eps \}_{\eps \in (0,\eps_0)}$ is bounded in $\L^2((0,T); \H^1(B_R))$.
		Using~\eqref{eq:cor_Loo2.5}, it also follows that $\{ \partial_t u_\eps \}_{\eps \in (0,\eps_0)}$ is bounded in $\L^2((0,T); \H^{-1}(B_R))$.
		Therefore, since the embedding $\H^1(B_R) \subseteq \L^2(B_R)$ is compact and the embedding $\L^2(B_R) \subseteq \H^{-1}(B_R)$ is continuous, the Aubin--Lions lemma~\cite{Simon_compactness} yields pre-compactness of $\{ u_\eps \}_{\eps \in (0,\eps_0)}$ in $\L^2(Q_T)$.
		Doing this with $T_m = m$ and $R_m = m$, $m\geq 1$, there exists a sequence $(\eps_k)_{k\geq 1}$ in $(0,\eps_0)$ and $u \in \L^2_\loc([0,\infty); \H^1_\loc(\R^n))$ such that
		\begin{itemize}
			\item $\nabla u_{\eps_k} \to \nabla u \quad$ weakly in $\L^2_\loc([0,\infty) \times \R^n)$,
			\item $u_{\eps_k} \to u \quad$ strongly in $\L^2_\loc([0,\infty) \times \R^n)$,
			\item $u_{\eps_k} \to u \quad$ almost everywhere on $(0,\infty) \times \R^n$.
		\end{itemize}
		For brevity, set $A(t,x) \coloneqq a(t,x,u(t,x))$ on $(0,\infty) \times \R^n$.
		By Assumption~\ref{ass:a}\ref{it:a4}, we have $A_{\eps_k} \to A$ almost everywhere on $(0,\infty) \times \R^n$.
		Hence, using weak convergence of $\nabla u_{\eps_k}$ we obtain
		\begin{align}
			\iint_{(0,\infty) \times \R^n} A_{\eps_k} \nabla u_{\eps_k} \cdot \nabla \phi \d x \d t \to \iint_{(0,\infty) \times \R^n} A \nabla u \cdot \nabla \phi \d x \d t, \quad \phi \in \Cont_\cc^\infty([0,\infty) \times \R^n).
		\end{align}
		By the dominated convergence theorem,
		\begin{align}
			\iint_{(0,\infty) \times \R^n} u_{\eps_k} (-\partial_t \phi) \d x \d t \to \iint_{(0,\infty) \times \R^n} u (-\partial_t \phi) \d x \d t, \quad \phi \in \Cont_\cc^\infty([0,\infty) \times \R^n).
		\end{align}
		Therefore, using~\eqref{eq:cor_Loo2.5}, $u$ solves $\partial_t u -\Div(A\nabla u) = 0$ weakly. Moreover, $\| u \|_{\L^\infty} \leq \| u_0 \|_\infty$ and $\overline{\vphantom{t}\Ran}(u) \subseteq O$ follow from~\eqref{eq:cor_Loo3} and~\eqref{eq:cor_Loo4} by almost everywhere convergence of $u_{\eps_k}$ to $u$.
		To establish that $u$ is a bounded weak solution to~\eqref{eq:ql}, it only remains to show $u(0) = u_0$. Let $\psi \in \Cont_\cc^\infty(\R^n)$ and let $\phi \in \Cont_\cc^\infty([0,\infty) \times \R^n)$ be any extension of $\psi$.
		Use the above limits and apply Lemma~\ref{lem:Lions_cont} twice, to obtain
		\begin{align}
			\langle u(0), \psi \rangle &= \iint_{(0,\infty) \times \R^n} u(-\partial_t \phi) + A\nabla u \cdot \nabla \phi \d x \d t \\
			&= \lim_k \iint_{(0,\infty) \times \R^n} u_{\eps_k}(-\partial_t \phi) + A\nabla u_{\eps_k} \cdot \nabla \phi \d x \d t = \lim_k \, \langle u_{\eps_k}(0), \psi \rangle.
		\end{align}
		Since $\langle u_{\eps_k}(0), \psi \rangle = \langle u_0 \ast \Phi_{\eps_k}, \psi \rangle \to \langle u_0, \psi \rangle$ by convergence of the mollified sequence, we find $u(0) = u_0$ as desired.
	\end{proof}

	\appendix

	\section{Carleson measure estimate}
	\label{sec:carleson}

	In this appendix, we provide an elementary proof for Proposition~\ref{prop:apriori}\ref{it:Carleson}.
	The result is a special case of~\cite[Thm.~7.4]{AMP19}, taking its localization to finite time intervals (explained in~\cite[Sec.~10]{AMP19}) into account.

	\begin{proposition}
		\label{prop:carleson}
		Let $u \colon (0,T) \times \R^n \to \R$ be a bounded weak solution to~\eqref{eq:CP_NAT}. Then, there exists a finite constant $C$ depending on the coefficient bounds of $A$, such that we have the Carleson measure estimate
		\begin{align}
			\sup_{x\in \R^n} \sup_{t \in (0,T]} \bigg( \int_0^t \fint_{\B(x,\sqrt{t})} |\nabla u(s,y)|^2 \d y \d s \bigg)^\frac{1}{2} \leq C \| u \|_{\L^\infty(T)}.
			\end{align}
	\end{proposition}

	\begin{proof}
		Fix $x\in \R^n$ and $t \in (0,T]$. For convenience, put $B \coloneqq \B(x,\sqrt{t})$ and let $2B$ be the concentric ball of radius $2\sqrt{t}$. Also, let $\eps \in (0, \nicefrac{t}{2})$. Since $u$ is a weak solution, we have $\nabla u \in \L^2((\eps,t) \times 2B)$. Hence,~\cite[(3.1)]{AMP19} applied with $a \coloneqq \eps$, $b \coloneqq t$ and $a' \coloneqq 2\eps$ gives
		\begin{align}
			\int_{2\eps}^t \| \nabla u(s,\cdot) \|_{\L^2(B)}^2 \d s \lesssim \| u(2\eps, \cdot) \|_{\L^2(2B)}^2 + \int_{2\eps}^t t^{-1} \| u(s, \cdot) \|_{\L^2(2B)}^2 \d s,
		\end{align}
		where the implicit constant depends on the coefficient bounds of $A$. It follows
		\begin{align}
			\int_{2\eps}^t \int_B |\nabla u(s,y)|^2 \d y \d s \lesssim t^\frac{n}{2} \| u \|_{\L^\infty(T)}^2.
		\end{align}
		We divide this bound by $t^{\nicefrac{n}{2}} \approx |B|$ and use Fatou's lemma, to find
		\begin{align}
			\int_0^t \fint_{B} |\nabla u(s,y)|^2 \d y \d s \leq \liminf_{\eps \to 0} \int_{2\eps}^t \fint_{B} |\nabla u(s,y)|^2 \d y \d s \lesssim \| u \|_{\L^\infty(T)}^2.
		\end{align}
		Taking the square root of this estimate lets us conclude.
	\end{proof}

\end{document}